\documentclass[preprint,12pt]{elsarticle}
\UseRawInputEncoding
\usepackage{amssymb}
\usepackage{amscd,amssymb,latexsym,srcltx}
\usepackage{dsfont}
\usepackage{amsthm}
\usepackage{setspace}
\usepackage{mdwlist}
\usepackage{enumerate}
\usepackage{graphics}
\usepackage{graphicx}
\usepackage{epsfig}
\usepackage{subfigure}
\usepackage{amssymb}
\usepackage{mathrsfs}
\usepackage{amsmath}
\allowdisplaybreaks
\usepackage{amsfonts}
\usepackage{geometry}
 \usepackage{hyperref}
\usepackage{geometry}
\biboptions{sort&compress}
\numberwithin{equation}{section}
\makeatother

\newtheorem{lemma}{Lemma}[section]

\newtheorem{theorem}{Theorem}[section]
\newtheorem{remark}{Remark}[section]
\newtheorem*{acknowledgments*}{ACKNOWLEDGMENTS}
\newtheorem{definition}{Definition}[section]

\newtheorem*{thm*}{Wedderburn-Artin's Theorem}
\newtheorem*{remark*}{Remark}
\newtheorem*{condition*}{Condition (C) }
\newtheorem*{lemma**}{Lemma 2}
\newtheorem*{example*}{[Example]}

\journal{}

\begin{document}

\begin{frontmatter}

\title{Estimate of the attractive velocity of attractors for some dynamical systems\tnoteref{t1}}
\tnotetext[t1]{Published on  Sci. Sin. Math., 52 (2022), 1--20, doi: 10.1360/SCM-2021-0470 (in Chinese).}

\author[label1]{Chunyan Zhao}
\ead{emmanuelz@163.com}
\author[label1]{Chengkui Zhong\corref{cor1}}
\ead{ckzhong@nju.edu.cn}
\author[label3]{Chunxiang Zhao}
\ead{zhaocxmath@163.com}
\address[label1]{Department of Mathematics, Nanjing University, Nanjing, 210093, China}
\address[label3]{Institute of Applied System Analysis, Jiangsu University,
Zhenjiang, 212013, China}

\cortext[cor1]{Corresponding author.}
\begin{abstract}
In this paper, we first prove an abstract theorem on the existence of polynomial attractors and the concrete estimate of their attractive velocity for infinite-dimensional dynamical systems, then apply this theorem to a class of wave equations with nonlocal weak damping and anti-damping in case that the nonlinear term~$f$~is of subcritical growth.
\end{abstract}


\end{frontmatter}

\section{Introduction}
In this paper, we first prove an abstract theorem on the existence of polynomial attractors and the concrete estimate of their attractive velocity for infinite-dimensional dynamical systems by estimating polynomial decay rate of noncompactness measure of bounded sets,
 then apply this theorem to a class of wave equations with nonlocal weak damping and anti-damping in case that the nonlinear term~$f$~is of subcritical growth. As far as we know, our work in this paper is the first one on polynomial attractors and polynomial decay with respect to noncompactness measure in the study of infinite-dimensional dynamical system.

It is a significant issue to  explore the long-term behavior of infinite-dimensional dynamical systems. The global attractor (see Definition~$\ref{def20-11-11-1}$)
is a core concept on this subject. By definition, if the global attractor exists, then it covers all possible permanent regimes of the system. Due to the~H\"{o}lder-Ma\~{n}\'{e} theorem (see \cite{Foias1996,MR1710097}), each
compact set with finite fractal dimension is homeomorphic to a compact subset of Euclidean space~$\mathds{R}^n$.
Therefore, if the fractal dimension of the
global attractor is finite, then the infinite-dimensional dynamical system restricted to the global attractor can be reduced to a finite-dimensional dynamical system (see \cite{MR2508165}).

However, as a characterization of the asymptotic dynamic behavior of the system, the global attractor has its limitations.
First, it may attract the trajectories at a low rate. Furthermore, it is very difficult to express the convergence rate in terms of the physical parameters. Second, a mathematical model is only an approximation
of the reality and thus an ideal mathematical object should be robust
under small perturbations. Nevertheless, the global attractor may be sensitive to perturbations due to its low attractive speed. In general, the global attractor is upper semicontinuous with respect to perturbations, whereas its lower semicontinuity is much more difficult to obtain. Finally, when the system has a global attractor
with finite fractal dimension, the reduced finite-dimensional dynamical system given by the H\"{o}lder-Ma\~{n}\'{e} theorem is only H\"{o}lder continuous, not necessarily Lipschitz continuous, so it is not necessarily generated by ordinary differential equations (see \cite{MR2508165,Robinson2001}).

 In order to overcome these limitations, Foias, Sell and Temam\cite{Inertialmanifolds} proposed the notion of inertial manifold in 1988. It is defined as a positively invariant finite-dimensional Lipschitz manifold which uniformly exponentially attracts all orbits starting from bounded subsets and contains the global attractor. The dynamical system restricted to the inertial manifold can be reduced to a Lipschitz continuous system of ordinary differential equations which is called the inertial form of the original system. Almost all known methods of constructing inertial manifolds are based on the ``spectral interval" condition (see \cite{Inertialmanifolds}), which is difficult to verify. The existence of inertial manifolds has been proved for a large number of equations with space dimension  1 or 2 (see\cite{MR966192,Inertialmanifolds,Robinson2001,MR1873467,TEMAM}). However, its existence is still an open
problem for several important equations, such as the two-dimensional incompressible Navier-Stokes equations. Furthermore, its nonexistence  has been proved for damped Sine-Gordon equations\cite{MR921912}.

As it is not always possible to obtain inertial manifolds, Eden, Foias, Nicolaenko and Temam\cite{MR1335230} proposed the concept of exponential attractor in 1994. An exponential attractor is a positively invariant, finite fractal dimensional compact set which uniformly exponentially attracts all orbits starting from bounded subsets and contains the global attractor. In other words, compared with inertial manifolds, smoothness is no longer required in the definition of exponential attractors. Actually, to the best of our knowledge, exponential attractors exist indeed for almost all equations with finite-dimensional global attractors.

A necessary prerequisite for a dynamical system to have an exponential attractor is that it possesses a global attractor with finite fractal dimension. On the other hand, many dynamical systems generated by evolution equations have infinite-dimensional global attractors, such as the~$p$-Laplace equations with symmetry (see \cite{MR2728546}), some reaction-diffusion equations in unbounded domains (see\cite{MR1871475,Zelik2003}), some hyperbolic equations in unbounded domains (see \cite{Zelik2001}), and so on. When a dynamical system has an infinite-dimensional global attractor, it has no exponential attractors, but may still
 have positively invariant and exponentially attractive compact sets. Based on this observation, Zhang, Kloeden, Yang and Zhong(\cite{Zhangjin1}) thought that the properties of exponential attractiveness and finite fractal dimension should be discussed separately, and proposed the concept of exponential decay with respect to noncompactness measure for the first time. They have proved that a sufficient and necessary condition for a dissipative dynamical system to have a positively invariant and exponentially attractive compact  set~$\mathcal{A}^*$~is that the noncompactness measure of bounded sets decays exponentially. They also gave some criteria for exponential decay with respect to noncompactness measure and proved this property for a class of reaction-diffusion equations and a class of wave equations with weak damping via the~$(C^{*})$~condition.


Inspired by the above literature, we attempt to estimate the attractive velocity of the attractor for the degenerate wave equation
 \begin{equation}\label{20-10-14-1}
u_{tt}-\Delta u+k||u_{t}||^p u_t+f(u)
=\displaystyle\int_{\Omega}K(x,y)u_{t}(y)dy+h(x)
\end{equation}
by studying the decay rate of the noncompactness measure. Since we cannot obtain the exponential decay estimate with respect to noncompactness measure for this equation, so we conjecture that it can reach the polynomial decay rate and the system may correspondingly have a polynomial attractor. This is the main motivation for the present paper.

 With this problem in mind, we noticed that M. Nakao\cite{zbMATH03541237,nakao1976,MR513082} proved that the nonnegative function~$w(t)$~which satisfies the difference inequality
\begin{equation*}
  \sup_{s\in[t,t+1]}w(s)^{1+\alpha}\leq K_0(1+t)^{\gamma}\big(w(t)-w(t+1)\big)+g(t)
\end{equation*}
decays to zero at polynomial or logarithm-polynomial rate as~$t\rightarrow+\infty$.
Subsequently,  many mathematicians have used Nakao's inequality to obtain the polynomial or logarithm-polynomial decay  estimate of the energy functional for different evolution equations(see \cite{nakao1977,MR513082,SilvaNarcisoVicente} for example).

 In this paper, we establish a quasi-stable inequality concerning the controlling relationship of the distance at time $t$ and the  initial distance between any two orbits  starting from a positive invariant bounded absorbing set~$\mathcal{B}_0$. This quasi-stable inequality is closely related to a difference inequality and contains compact pseudo-metrics. Thus, by the definition of noncompactness measure and the compactness of the pseudo-metrics, from this quasi-stable inequality we can deduce the difference inequality concerning the noncompactness measure $\alpha(S(t)\mathcal{B}_0)$,  which  leads to the  estimate of the polynomial decay rate  with respect to the noncompactness measure.  Consequently, the existence of polynomial attractors and the estimate of their attractive rate are obtained.

When applying the abstract theorem to the concrete wave equation,  we use the strong monotone inequality in inner product space, the property of concave function and the method of energy reconstruction in the process of asymptotic estimation of energy to verify the quasi-stable condition. The energy estimation formula thus obtained naturally contains the difference item ~$E_z(0)-E_z(T)$  (where ~$E_z(t)$ denotes the energy of the difference~$z$ between two solutions at time~$t$), and consequently leads to the quasi-stable inequality closely related to the difference inequality.

As for the study of the attractive velocity of attractors for infinite dimensional dynamical systems, the existing literature has been limited to the exponential attractive velocity. Meanwhile, there have  been many results on the polynomial decay rate of energy functionals. To our best knowledge, we are the first to study polynomial attractors.

As regard to the approach, we combine the energy estimate, energy reconstruction, noncompactness measure and a kind of difference inequality leading to the  estimate of the decay rate of non-negative functions, while the traditional method is to construct appropriate Gronwall's inequalities by energy estimate.

The paper is organized as follows. We give some necessary preliminaries in Section 2. We estabilish an abstract theorem on the existence of polynomial attractors and the  explicit  estimate of their attractive velocity for some infinite-dimensional dynamical systems in Section 3, and then apply this abstract theorem to a class of wave equations with subcritical nonlinearity in Section 4.

 Throughout this paper,~$\Omega\subseteq \mathds{R}^{n}$ denotes a bounded domain with a sufficiently smooth boundary~$\partial\Omega$. We will denote the inner product and the norm on~$L^2(\Omega)$~by~$(\cdot,\cdot)$ and $\|\cdot\|$ respectively, and the norm on~$L^p(\Omega)$~by~$\|\cdot\|_p$. The symbols~$\hookrightarrow$~and~$\hookrightarrow\hookrightarrow$~stand for continuous embedding and compact embedding respectively.
The capital letter ``C" with a (possibly empty) set of subscripts will denote
a positive constant depending only on its subscripts and may vary from one
occurrence to another.
\section{Preliminaries}
In this section, we give some necessary preliminaries which are required for establishing
our results.

 We first briefly recall the definition of Kuratowski measure of noncompactness and its basic properties. For more details, we refer to \cite{K. Deimling,MR1153247}.

\begin{definition}\cite{K. Deimling,MR1153247}\label{def20-10-26-1}
 Let~$(X,d)$~be a metric space and let~$B$~be a bounded subset of~$X$.
The Kuratowski~$\alpha$-measure of noncompactness is defined by
\begin{equation*}
  \alpha(B) = \inf\{\delta>0 | B\   \text{has a finite cover of diameter} ~<\delta\}.
\end{equation*}
\end{definition}

\begin{lemma}\cite{K. Deimling,MR1153247}
Let~$(X,d)$~be a complete metric space and~$\alpha$~be the Kuratowski measure of noncompactness. Then
\begin{enumerate}[(i)]
  \item~$\alpha(B)=0 $~if and only if~$B$~is precompact;
  \item~$\alpha(A)\leq\alpha(B)$~whenever~$A\subseteq B$;
  \item~$\alpha(A \cup B)=\max\{\alpha(A), \alpha(B)\}$;
  \item~$\alpha(B)=\alpha(\overline{B})$, where~$\overline{B}$~is the closure of~$B$;
  \item~if~$B_1\supseteq B_2\supseteq B_3\ldots$~are nonempty closed sets in~$X$~such that~$\alpha(B_n)\rightarrow 0$~as~$n\rightarrow \infty$, then~$\cap_{n\geq1}B_n$~is nonempty and compact;
\item~if~$X$~is a Banach space, then~$\alpha(A+B)\leq \alpha(A)+\alpha(B)$.
\end{enumerate}
\end{lemma}

Next, we will briefly review the definitions and fundamental conclusions of dynamical systems and the global attractor.
\begin{definition}\cite{Robinson2001}
  A dynamical system is a pair of objects~$(X, \{S(t)\}_{t\geq0})$~consisting of a
complete metric space~$X$~and a family of continuous mappings~$\{S(t)\}_{t\geq0}$~of X
into itself with the semigroup properties:
\begin{enumerate}[(i)]
  \item ~$S(0)=I$,
  \item ~$S(t+s)=S(t)S(s)$~for all~$t,s\geq0$,
\end{enumerate}
where~$X$~is called a phase
space (or state space) and~$\{S(t)\}_{t\geq0}$~is called an evolution semigroup.
\end{definition}

\begin{definition}\cite{Robinson2001}
Let~$\{S(t)\}_{t\geq0}$~be a semigroup on a complete metric space~$(X,d)$. A closed set~$\mathcal{B}\subseteq X$~is said to be absorbing for~$\{S(t)\}_{t\geq0}$~iff for any bounded
set~$B\subseteq X$~there exists~$t_0(B)$~$($the entering time of~$B$~into~$\mathcal{B}$$)$ such that~$S(t)B\subseteq \mathcal{B}$~for all~$t>t_0(B)$.
 $\{S(t)\}_{t\geq0}$~is said to be dissipative iff it possesses a bounded absorbing set.
\end{definition}

\begin{lemma}\cite{Chueshov2008}\label{20-5-5-2}
Let~$\{S(t)\}_{t\geq0}$~be a semigroup on a complete metric space~$(X,d)$. If~$\{S(t)\}_{t\geq0}$ is dissipative, then it possesses a positively invariant bounded absorbing set. To be more precise, let~$ \mathcal{B}$~be its bounded absorbing set, then~$ \mathcal{B}_0=\overline{\bigcup\limits_{t\geq t_{\mathcal{B}}}S(t)\mathcal{B}}$~is a positively invariant bounded absorbing set, where~$t_{ \mathcal{B}}>0$~is the entering time of~$ \mathcal{B}$~into itself.
\end{lemma}

\begin{definition} \cite{Robinson2001}\label{def20-11-11-1}
A compact set~$\mathcal{A}\subseteq X$~is said to be a global attractor of the dynamical system~$(X, \{S(t)\}_{t\geq0})$~iff
\begin{enumerate}[(i)]
  \item~$\mathcal{A}\subseteq X$~is an invariant set, i.e.,~$S(t)\mathcal{A}=\mathcal{A}$~for all~$t\geq0$,
  \item~$\mathcal{A}\subseteq X$~is uniformly attracting, i.e., for all bounded set~$B\subseteq X$~we have
  \begin{equation*}
\lim_{t\rightarrow +\infty}\mathrm{dist}(S(t)B,\mathcal{A})=0,
\end{equation*}
where~$\mathrm{dist}(A,B):=\sup_{x\in A}\mathrm{dist}_{X}(x,B)$~is the Hausdorff semi-distance.
\end{enumerate}
\end{definition}

Ma, Wang and Zhong put forward the concept of~$\omega $-limit compact in \cite{MWZH} and proved that~$\omega $-limit compactness is a necessary and sufficient condition for a dissipative dynamical system to possess a global attractor.

\begin{definition}
The dynamical system~$(X, \{S(t)\}_{t\geq0})$~is called to be~$\omega$-limit compact iff for every positively invariant bounded set~$B\subseteq X$~we have~$\alpha\big(S(t)B\big)\rightarrow 0$~as~$t\rightarrow\infty$, where~$\alpha(\cdot)$~is the Kuratowski measure of noncompactness.
\end{definition}

\begin{theorem}\cite{MWZH}
The dynamical system~$(X, \{S(t)\}_{t\geq0})$~has a global attractor in~$X$~if and only if it is both dissipative and ~$\omega$-limit compact.
\end{theorem}

The global attractor gives no information about the attractive rate. In order to describe the asymptotic behavior of dynamical systems more concretely, we propose the following concept of~$\varphi$-attractor.

\begin{definition}\label{def21-1-31-1}
Assume that~$\varphi:\mathds{R}^{+}\rightarrow\mathds{R}^{+}$~satisfies~$\varphi(t)\rightarrow 0$~as~$t\rightarrow +\infty$. We call a compact~$\mathcal{A}^*\subseteq X$~a~$\varphi$-attractor for the dynamical system~$(X, \{S(t)\}_{t\geq0})$, iff~$\mathcal{A}^*$~is positively invariant with respect to~$S(t)$~and there exists~$t_{0}\in \mathds{R}$~such that for every bounded set~$B\subseteq X$~there exists~$t_{B}\geq0$~such that
\begin{equation*}
\begin{split}
\mathrm{ dist}\left(S(t)B, \mathcal{A}^*\right)\leq \varphi(t+t_0-t_{B}),\ \forall t\geq t_{B}.
  \end{split}
\end{equation*}
In particular, if~$\varphi(t)=Ct^{-\beta}$~for certain positive constants~$C,\beta$, then~$\mathcal{A}^*$~is called a polynomial attractor.
\end{definition}

We emphasize that the finiteness of fractal dimension is not required in the above definition of~$\varphi$-attractor. This is because there indeed exist positively invariant compact sets with infinite fractal dimension which can attract all bounded sets at the~$\varphi$-speed for some dynamical systems.

We further propose the following concept of~$\varphi$-decay with respect to noncompactness measure as a condition for the existence of~$\varphi$-attractor.

\begin{definition}\label{def20-8-5-3}
The dynamical system~$(X, \{S(t)\}_{t\geq0})$~is called to be~$\varphi$-decaying with respect to noncompactness measure~$\alpha$ iff it is dissipative and there exists~$t_{0}>0$~such that
  \begin{equation}\label{20-8-5-1}
  \alpha(S(t)\mathcal{B}_0)\leq\varphi(t),\ \forall t\geq t_{0},
\end{equation}
where~$\mathcal{B}_0$~is a positively invariant bounded absorbing set of~$(X, \{S(t)\}_{t\geq0})$~and~$\varphi:\mathds{R}^{+}\rightarrow\mathds{R}^{+}$~is a decreasing function satisfying~$\varphi(t)\rightarrow 0$~as~$t\rightarrow +\infty$.

In particular, if~$\varphi(t)=Ct^{-\beta}$~$($or~$\varphi(t)=Ce^{-\beta t}$$)$ for certain positive constants~$C,\beta$, then~$(X, \{S(t)\}_{t\geq0})$~is said to be polynomially decaying $($or be exponentially decaying$)$ with respect to noncompactness measure~$\alpha$.
\end{definition}

The following  lemma shows that the decay rate of  noncompactness measure of a positively invariant bounded absorbing set determines the decay rate of noncompactness measure of any bounded set, which guarantees  the justification of Definition~$\ref{def20-8-5-3}$.

\begin{lemma}\label{lemma 8-20-5-12}
 Assume that the dynamical system~$(X, \{S(t)\}_{t\geq0})$~is~$\varphi$-decaying with respect to noncompactness measure $\alpha$, which implies that there exist a positively invariant bounded absorbing set~$\mathcal{B}_0$~and a positive constant~$t_{0}$~such that
\begin{equation*}
  \alpha(S(t)\mathcal{B}_0)\leq\varphi(t),\ \forall t\geq t_{0}.
\end{equation*}
Then for every bounded subset~$B$~of~$X$, we have
\begin{equation*}
\begin{split}
  \alpha(S(t)B)\leq\varphi(t-t_{*}(B)),\ \forall t\geq t_{*}(B)+t_{0},
  \end{split}
\end{equation*}
where~$t_{*}(B)$~is the entering time of~$B$~into~$\mathcal{B}_0$.
\end{lemma}

\begin{proof}It follows from
\begin{equation*}
  S(t)B\subseteq \mathcal{B}_0, \ \forall t\geq t_{*}(B)
\end{equation*}
that~$S(t)B=S(t-t_{*}(B))S(t_{*}(B))B\subseteq S(t-t_{*}(B))\mathcal{B}_0$.
Consequently, we have
\begin{equation*}
\begin{split}
\alpha(S(t)B)\leq\alpha(S(t-t_{*}(B))\mathcal{B}_0)\leq\varphi(t-t_{*}(B)),\ \forall t\geq t_{*}(B)+t_0.
  \end{split}
\end{equation*}
\end{proof}

The importance of~$\varphi$-decay with respect to noncompactness measure is reflected in the following theorem.
\begin{theorem}\label{20-8-5-3}
Assume that the dynamical system~$(X, \{S(t)\}_{t\geq0})$~is~$\varphi$-decaying with respect to noncompactness measure $\alpha$, which implies that there exist a positively invariant bounded absorbing set~$\mathcal{B}_0$~and a positive constant~$t_{0}$~such that
\begin{equation*}
  \alpha(S(t)\mathcal{B}_0)\leq\varphi(t),\ \forall t\geq t_{0}.
\end{equation*}
Then there exists a positively invariant compact set~$\mathcal{A}^*$~such that for every bounded set~$B\subseteq X$~we have
\begin{equation}\label{20-8-5-33}
\begin{split}
\mathrm{ dist}\left(S(t)B, \mathcal{A}^*\right)\leq \varphi(t-t_{*}(B)-1),\ \forall t\geq t_{*}(B)+t_{0}+1,
  \end{split}
\end{equation}
where~$t_{*}(B)$~is the entering time of~$B$~into~$\mathcal{B}_0$. In other words,~$(X, \{S(t)\}_{t\geq0})$~possesses a~$\varphi$-attractor(see Definition~\ref{def21-1-31-1}).
\end{theorem}
\begin{proof}
\textbf{Step1:}
It follows from
\begin{equation*}
\left\{
\begin{array}{rl}
&\varphi(t)\rightarrow 0\  \text{as}\ t\rightarrow +\infty,\\
&\alpha(S(t)\mathcal{B}_0)\leq\varphi(t),\ \forall t\geq t_{0}
\end{array}
\right.
\end{equation*}
that
\begin{equation}\label{20-8-5-21}
\alpha(S(t)\mathcal{B}_0)\rightarrow 0 (t\rightarrow +\infty),
\end{equation}
from which we can defuce that~$\mathcal{A}=\omega (\mathcal{B}_0)\equiv \bigcap_{t\geq 0}\overline{S(t)\mathcal{B}_0}$ is the global attractor of this system.

Since
\begin{equation*}
  \alpha(S(t)\mathcal{B}_0)\leq\varphi(t),\ \forall t\geq t_{0},
\end{equation*}
for every positive integer~$m\geq t_{0}$,~$S(m)\mathcal{B}_0$~has a finite~$\varphi(m)$-net which we denote by $E_m=\bigcup_{i=1}^{K_m}S(m)a_i^{(m)},\  a_i^{(m)}\in \mathcal{B}_0$, i.e.,~$S(m)\mathcal{B}_0\subseteq \bigcup_{x_{\lambda}\in E_m}B(x_{\lambda},\varphi(m))$.

Let~$F_m=\bigcup\limits_{t\geq0}S(t)E_m=\bigcup\limits_{t\geq0}S(t)\bigcup\limits_{i=1}^{K_m}S(m)a_i^{(m)},\  a_i^{(m)}\in \mathcal{B}_0$. Take~$\mathcal{A}^*=\bigcup\limits_{m\in \mathds{N},~m\geq t_{0}}F_m\bigcup\mathcal{A}$. It is obvious that $\mathcal{A}^*$ is positively invariant.

Since~$\mathcal{A}^*\supseteq\bigcup_{m=[t_0]+1}^{+\infty}E_m$,
\begin{equation}\label{20-8-5-2}
  \mathrm{dist}\left(S(m)\mathcal{B}_0, \mathcal{A}^*\right)\leq \varphi(m), \ \forall m\geq t_{0}.
\end{equation}
By the positive invariance of~$\mathcal{B}_0$,~$S(t)\mathcal{B}_0\subseteq S([t])\mathcal{B}_0$~(where~$[t]$~is the integral part of~$t$~). Hence we deduce from the monotonicity of $\varphi(t)$ and $(\ref{20-8-5-2})$ that
\begin{equation}\label{20-8-5-4}
\begin{split}
 \mathrm{dist}\left(S(t)\mathcal{B}_0, \mathcal{A}^*\right)\leq &dist\left(S([t])\mathcal{B}_0, \mathcal{A}^*\right)\\
 \leq &\varphi([t])\\
 \leq &\varphi(t-1)
  \end{split}
\end{equation}
holds for all~$t\geq t_{0}+1$.
For every bounded set~$B\subseteq X$, there exists~$t_{*}(B)$ such that
\begin{equation*}
  S(t)B\subseteq \mathcal{B}_0, \ \forall t\geq t_{*}(B).
\end{equation*}
Therefore~$S(t)B=S(t-t_{*}(B))S(t_{*}(B))B\subseteq S(t-t_{*}(B))\mathcal{B}_0$, and thus for every $t\geq t_{*}(B)+t_{0}+1$ we have
\begin{equation}\label{20-8-5-5}
\begin{split}
 \mathrm{dist}\left(S(t)B, \mathcal{A}^*\right)\leq &\mathrm{dist}\left(S(t-t_{*}(B))\mathcal{B}_0, \mathcal{A}^*\right)\\
 \leq &\varphi(t-t_{*}(B)-1).
  \end{split}
\end{equation}

\textbf{Step2:}~Next we shall verify that~$\mathcal{A}^*$ is compact, i.e., every sequence $\{x_n\}_{n=1}^{+\infty}\subseteq\mathcal{A}^*$ has a subsequence which converges to a point in $\mathcal{A}^*$.
Since the global attractor~$\mathcal{A}$~is compact, if there exists a subsequence $\{x_{n_k}\}_{k=1}^{+\infty}\subseteq\mathcal{A}$, then $\mathcal{A}^*$ is compact. Therefore, without loss of generality we can assume that
~$\{x_n\}_{n=1}^{+\infty}\subseteq\bigcup_{m\in \mathds{N},~m\geq t_{0}}F_m$. Write
\begin{equation}\label{20-8-5-10}
x_n=S(t_n)S(m_n)a_{i_n}^{(m_n)}, \ t_n\geq 0, m_n\geq t_{0}, 1\leq i_n\leq k_{m_n}, a_{i_n}^{(m_n)}\in \mathcal{B}_0.
\end{equation}
\begin{itemize}
  \item[(i)] If~$\{t_n+m_n\}_{n=1}^{+\infty}$ is unbounded, then there exists a subsequence~$\{n_k\}$~of~$\{n\}$ such that $t_{n_k}+m_{n_k}\rightarrow+\infty$ as $k\rightarrow+\infty$. Hence we can deduce from $(\ref{20-8-5-21})$ that there exists a subsequence of $\{x_n\}_{n=1}^{+\infty}$ convergent in~$\mathcal{A}=\omega(\mathcal{B}_0)\subseteq \mathcal{A}^*$.
  \item[(ii)] If there exists a positive integer~$N_0$ such that~$t_n+m_n\leq N_0$ for all $n\in \mathds{N}$, then~$\{x_n\}\subseteq \bigcup\limits_{m\in \mathds{N},~t_{0}\leq m\leq N_0 }\bigcup\limits_{i=1}^{K_m}\bigcup\limits_{t\in [0,N_0]}S(t)S(m)a_i^{(m)}$. For given~$m$ and~$i$, mapping $t\rightarrow S(t)S(m)a_i^{(m)}$ is continuous, and~$[0,N_0]$ is a compact set, so~$\bigcup\limits_{t\in [0,N_0]}S(t)S(m)a_i^{(m)}$ is compact, and thus $\bigcup\limits_{m\in \mathds{N},~t_{0}\leq m\leq N_0 }\bigcup\limits_{i=1}^{K_m}\bigcup\limits_{t\in [0,N_0]}S(t)S(m)a_i^{(m)}$ is also compact. Consequently, there exists a subsequence of $\{x_n\}_{n=1}^{+\infty}$ convergent in~$\bigcup\limits_{m\in \mathds{N}, t_{0}\leq m\leq N_0 }\bigcup\limits_{i=1}^{K_m}\bigcup\limits_{t\in [0,N_0]}S(t)S(m)a_i^{(m)}\subseteq\mathcal{A}^*$.
\end{itemize}

This finishes the proof.
\end{proof}
\begin{remark}
The method of constructing $\varphi$-attractors~$\mathcal{A}^*$ is to add a countable collection of positively invariant  point sets to the global attractor such that the added point sets attract a positively  invariant bounded absorbing set~$\mathcal{B}_0$ at $\varphi$-speed and thus can  attract any bounded set~$B$ at  $\varphi$-speed. The $\varphi$-decay  with respect to noncompactness measure guarantees that the union of the global attractor and these added point sets is compact. Therefore,  the decay rate with respect to noncompact measure $\alpha$ essentially  characterizes the attractive rate that a positively invariant compact set may reach.
\end{remark}


At the end of this section, we will give a lemma concerning the decay estimate of nonnegative functions. The proof of the main theorem in this paper is based on this lemma.
\begin{lemma}\label{20-7-26-40}
Suppose that~$\omega(t)$~is a nonnegative function on~$\mathds{R}^{+}$~satisfying
\begin{equation}\label{20-7-26-41}
  \max\{\omega^{1+\alpha}(t),\omega^{1+\alpha}(t+T)\}\leq h(t)[\omega(t)-\omega(t+T)],\ \forall t\geq t_0,
\end{equation}
where~$\alpha,T, t_0$~are positive constants,~$h(t)$~is a positive monotone function. Then~$\omega(t)$~satisfies the following estimate:
\begin{equation*}
\begin{split}
\omega(t)\leq \Big\{ \inf_{s\in[t_0,t_0+T]}\omega^{-\alpha}(s)+\frac{\alpha}{T}\int_{t_0+T}^{t-T}\frac{ds}{h(s)}\Big\}^{-\frac{1}{\alpha}},\ \forall t\geq t_0+2T.
 \end{split}
 \end{equation*}

In particular, when~$h(t)=K_0$, we have
\begin{equation}\label{20-7-30-5}
\begin{split}
\omega(t)\leq \Big\{ \inf_{s\in[t_0,t_0+T]}\omega^{-\alpha}(s)+\frac{\alpha}{TK_0}(t-t_0-2T)\Big\}^{-\frac{1}{\alpha}},\ \forall t\geq t_0+2T.
 \end{split}
 \end{equation}
\end{lemma}
\begin{proof}
\begin{equation}\label{20-7-26-43}
\begin{split}
&\omega^{-\alpha}(t+T)-\omega^{-\alpha}(t)\\=&\int_0^1\frac{d}{d\theta}\Big\{\big[\theta \omega(t+T)+(1-\theta)\omega(t)\big]^{-\alpha}\Big\}d\theta\\
=&\alpha\int_0^1\big[\theta \omega(t+T)+(1-\theta)\omega(t)\big]^{-\alpha-1}d\theta\big[\omega(t)-\omega(t+T)\big]\\
\geq &\alpha\big(\max\{\omega(t),\omega(t+T)\}\big)^{-\alpha-1}\big[\omega(t)-\omega(t+T)\big].
\end{split}
\end{equation}
It follows from~$(\ref{20-7-26-41})$~and~$(\ref{20-7-26-43})$~that
\begin{equation*}
\begin{split}
\omega^{-\alpha}(t+T)-\omega^{-\alpha}(t)\geq \frac{\alpha}{h(t)},\ \forall t\geq t_0.
 \end{split}\end{equation*}
 Thus
 \begin{equation}\label{20-7-30-1}
\begin{split}
\omega^{-\alpha}(t)\geq \omega^{-\alpha}(t-nT)+\frac{\alpha}{T}\sum_{i=1}^{n}\frac{T}{h(t-iT)},\ \forall t\geq t_0+T,
 \end{split}
 \end{equation}
where~$n\equiv[\frac{t-t_0}{T}]$~is the integral part of~$\frac{t-t_0}{T}$.

 If~$h(t)$~is non-increasing, then by~$(\ref{20-7-30-1})$,
 \begin{equation}\label{20-7-30-2}
\begin{split}
\omega^{-\alpha}(t)\geq &\omega^{-\alpha}(t-nT)+\frac{\alpha}{h(t-nT)}+\frac{\alpha}{T}\sum_{i=1}^{n-1}\int_{t-(i+1)T}^{t-iT}\frac{ds}{h(t-iT)}\\
\geq &\omega^{-\alpha}(t-nT)+\frac{\alpha}{h(t-nT)}+\frac{\alpha}{T}\int_{t-nT}^{t-T}\frac{ds}{h(s)}\\
\geq &\omega^{-\alpha}(t-nT)+\frac{\alpha}{T}\int_{t_0+T}^{t-T}\frac{ds}{h(s)}\\
\geq &\inf_{s\in[t_0,t_0+T]}\omega^{-\alpha}(s)+\frac{\alpha}{T}\int_{t_0+T}^{t-T}\frac{ds}{h(s)}
 \end{split}
 \end{equation}
holds for all~$t\geq t_0+2T$.

 If~$h(t)$~is non-decreasing, then by~$(\ref{20-7-30-1})$
 \begin{equation}\label{20-7-30-3}
\begin{split}
\omega^{-\alpha}(t)\geq &\omega^{-\alpha}(t-nT)+\frac{\alpha}{T}\sum_{i=1}^{n}\int_{t-iT}^{t-(i-1)T}\frac{ds}{h(t-iT)}\\
\geq &\omega^{-\alpha}(t-nT)+\frac{\alpha}{T}\int_{t-nT}^{t}\frac{ds}{h(s)}\\
\geq &\omega^{-\alpha}(t-nT)+\frac{\alpha}{T}\int_{t_0+T}^{t-T}\frac{ds}{h(s)}\\
\geq &\inf_{s\in[t_0,t_0+T]}\omega^{-\alpha}(s)+\frac{\alpha}{T}\int_{t_0+T}^{t-T}\frac{ds}{h(s)}
 \end{split}
 \end{equation}
holds for all~$t\geq t_0+2T$.

Combining~$(\ref{20-7-30-2})$~and~$(\ref{20-7-30-3})$, we conclude that if~$h(t)$~is monotone, then
\begin{equation}\label{20-7-30-4}
\begin{split}
\omega(t)\leq \Big\{ \inf_{s\in[t_0,t_0+T]}\omega^{-\alpha}(s)+\frac{\alpha}{T}\int_{t_0+T}^{t-T}\frac{ds}{h(s)}\Big\}^{-\frac{1}{\alpha}},\ \forall t\geq t_0+2T.
 \end{split}
 \end{equation}
The estimate~$(\ref{20-7-30-5})$~
follows immediately from~$(\ref{20-7-30-4})$.
 \end{proof}

\begin{remark}
Lemma~$\ref{20-7-26-40}$~estimates the decay rate of nonnegative function~$\omega(t)$~satisfying the difference inequality~$(\ref{20-7-26-41})$~in the case that~$h(t)$~is a general positive monotone function. It is a generalization of Theorem~$1$~in \cite{MR513082} by~M. Nakao. The latter established decay estimate from the
difference inequality
\begin{equation*}
  \sup_{s\in[t,t+1]}\omega(s)^{1+\alpha}\leq K_0(1+t)^{\gamma}\big(\omega(t)-\omega(t+1)\big)+g(t).
\end{equation*}
\end{remark}

\section{Abstract results on existence of polynomial attractors and estimate of their attractive velocity}

\begin{lemma}\label{lem21-4-18-1}
Let~$\{S(t)\}_{t\geq0}$ be a dissipative dynamical system on a complete metric space~$(X,d)$ and~$\mathcal{B}_0$ be a positively invariant bounded absorbing set. Assume that there exist positive constants~$T, \delta_0$, a continuous non-decreasing function $q:\mathds{R}^{+}\rightarrow \mathds{R}^{+}$, a function~$g:(\mathds{R}^{+})^{m}\rightarrow\mathds{R}^{+}$~and pseudometrics ~$\varrho_{T}^{i}\ (i=1,2,\ldots,m)$~on~$\mathcal{B}_0$~such that
\begin{itemize}
\item[(i)]~$q(0)=0$;~$q(s)<s,\ s>0$;
\item[(ii)]~$g$~is non-decreasing with respect to each
variable,~$g(0,\ldots,0)=0$~and~$g$~is continuous at~$(0,\ldots,0)$;
\item[(iii)]~$\varrho_{T}^{i}(i=1,2,\ldots,m)$~is precompact on $\mathcal{B}_0$, i.e., any sequence~$\{x_n\}\subseteq \mathcal{B}_0$~has a subsequence~$\{x_{n_k}\}$~which is Cauchy with
respect to~$\varrho_{T}^{i}$;
\item[(iv)]~the inequality
\begin{equation}\label{21-4-18-2}
\begin{split}
&\big(d(S(T)y_1,S(T)y_2)\big)^2\\\leq
&q\Big(\big(d(y_1,y_2)\big)^2+g\big(\varrho_{T}^1(y_1,y_2),\varrho_{T}^2(y_1,y_2),\ldots,\varrho_{T}^m(y_1,y_2)\big)\Big)
\end{split}
\end{equation}
 holds for all~$y_1,y_2\in \mathcal{B}_0$~satisfying
~$\varrho_{T}^{i}(y_1,y_2)\leq \delta_0(i=1,2,\ldots,m)$.
\end{itemize}
Then
\begin{equation*}
\alpha(S(t)\mathcal{B}_0)\rightarrow 0 \ \text{as}\ t\rightarrow +\infty,
\end{equation*}
and~$\{S(t)\}_{t\geq0}$~is~$\omega$-limit compact.
\end{lemma}
\begin{proof}
For each ~$B\subseteq \mathcal{B}_0$ and any $\epsilon>0$, by Definition~$\ref{def20-10-26-1}$, there exist sets~$F_1,F_2,\ldots,F_n$~such that
\begin{equation}\label{21-4-19-1}
  B\subseteq\cup_{j=1}^{n}F_n,\ \mathrm{diam} F_j<\alpha(B)+\epsilon.
\end{equation}
It follows from assumption (ii) that there exists
~$\delta>0$~such that~$g(x_1,x_2,\ldots,x_m)<\epsilon$~whenever~$x_i\in [0,\delta]\ (i=1,2,\ldots,m)$. By the precompactness of
~$\varrho_{T}^i(i=1,2,\ldots,m)$, there exists a finite set ~$\mathcal{N}^i=\{x^i_{j}:j=1,2,\ldots,k_i\}\subseteq B$~such that for every~$y\in B$~there is~$x^i_{j}\in \mathcal{N}^i$~with the property ~$\varrho_{T}^i(y,x^i_{j})\leq\frac{1}{2}\min\{\delta,\delta_0\}$, i.e.,
\begin{equation}\label{21-4-19-9}
B\subseteq\cup_{j=1}^{k_i}C_{j}^{i},\ C_{j}^{i}=\big\{y\in B:\ \varrho_{T}^i(y,x^i_{j})\leq\frac{1}{2}\min\{\delta,\delta_0\}\big\},\ i=1,\ldots,m.
\end{equation}
Consequently, we have
\begin{equation*}
  B\subseteq\cup_{j_1,j_2,\ldots,j_m,j}( C_{j_1}^{1}\cap C_{j_2}^{2}\cap\ldots\cap C_{j_m}^{m}\cap F_j)
\end{equation*}
and
\begin{equation*}
  S(T)B\subseteq\cup_{j_1,j_2,\ldots,j_m,j}\big(S(T)( C_{j_1}^{1}\cap C_{j_2}^{2}\cap\ldots\cap C_{j_m}^{m}\cap F_j)\big).
\end{equation*}
By~$(\ref{21-4-19-1})$~and~$(\ref{21-4-19-9})$, for any~$y_1,y_2\in C_{j_1}^{1}\cap C_{j_2}^{2}\cap\ldots\cap C_{j_m}^{m}\cap F_j$, we have
\begin{equation}\label{21-4-19-11}
  d\big(y_1,y_2\big)\leq \mathrm{diam} F_j<\alpha(B)+\epsilon
\end{equation}
and
\begin{equation}\label{21-4-19-12}
  \varrho_{T}^i\big(y_1,y_2\big)\leq \min\{\delta,\delta_0\}, i=1,2,\ldots,m.
\end{equation}

Inequality~$(\ref{21-4-19-12})$~implies
\begin{equation}\label{21-4-19-13}
  g\Big(\varrho_{T}^1\big(y_1,y_2\big),\ldots,\varrho_{T}^m\big(y_1,y_2\big)\Big)<\epsilon.
\end{equation}
We deduce from~$(\ref{21-4-18-2})$, $(\ref{21-4-19-11})$, $(\ref{21-4-19-12})$~and~$(\ref{21-4-19-13})$~that
\begin{equation}\label{21-4-19-14}
\big(d(S(T)y_1,S(T)y_2)\big)^2\leq
q\Big(\big(\alpha(B)+\epsilon\big)^2+\epsilon\Big)
\end{equation}
for any~$y_1,y_2\in C_{j_1}^{1}\cap C_{j_2}^{2}\cap\ldots\cap C_{j_m}^{m}\cap F_j$. Therefore according to the definition of noncompactness measure~$\alpha$, we obtain
\begin{equation}\label{21-4-19-15}
\big(\alpha(S(T)B)\big)^2\leq
q\Big(\big(\alpha(B)+\epsilon\big)^2+\epsilon\Big).
\end{equation}
Since~$q$ is continuous and non-decreasing, combining $(\ref{21-4-19-15})$ and the arbitrariness of~$\epsilon$ gives
\begin{equation}\label{21-4-19-16}
\big(\alpha(S(T)B)\big)^2\leq
q\Big(\big(\alpha(B)\big)^2\Big).
\end{equation}
We infer from $(\ref{21-4-19-16})$ that
\begin{equation}\label{21-4-19-17}
\big(\alpha(S(kT)\mathcal{B}_0)\big)^2\leq
q\Big(\big(\alpha(S((k-1)T)\mathcal{B}_0)\big)^2\Big),\ k=1,2,\ldots.
\end{equation}
Since~$q(s)\leq s$, the sequence $\Big\{\Big(\alpha\big(S(kT)\mathcal{B}_0\big)\Big)^2\Big\}_{k=1}^{+\infty}$ is non-increasing and thus there exists~$\alpha_0=\lim_{k\rightarrow+\infty}\Big(\alpha\big(S(kT)\mathcal{B}_0\big)\Big)^2$.
By the continuity of~$q$, $(\ref{21-4-19-17})$ implies~$\alpha_0\leq q(\alpha_0)$, which yields~$\alpha_0=0$~by assumption (ii). Consequently, we obtain
\begin{equation}\label{21-6-25-1}
\alpha(S(t)\mathcal{B}_0)\rightarrow 0 \ \text{as}\ t\rightarrow +\infty.
\end{equation}
For any bounded set $D\subseteq X$, there exists $t_D>0$ such that $S(t_D)D\subseteq\mathcal{B}_0$, which implies $S(t+t_D)D\subseteq S(t)\mathcal{B}_0$. Thus, we have $\alpha\big(S(t+t_D)D \big)\leq \alpha\big(S(t)\mathcal{B}_0 \big)$, which, together with $(\ref{21-6-25-1})$, gives
\begin{equation}\label{21-4-19-20}
\alpha(S(t)D)\rightarrow 0 \ \text{as}\ t\rightarrow +\infty.
\end{equation}
Therefore, $\{S(t)\}_{t\geq0}$~is~$\omega$-limit compact.
\end{proof}

\begin{theorem}\label{20-7-27-80}
Let~$\{S(t)\}_{t\geq0}$ be a dissipative dynamical system on a complete metric space~$(X,d)$ and~$\mathcal{B}_0$ be a positively invariant bounded absorbing set. Assume that there exist positive constants~$C, T, \delta_0$,~$\beta\in(0,1)$, functions~$g_l:(\mathds{R}^{+})^{m}\rightarrow\mathds{R}^{+}\ (l=1,2)$ and pseudometrics~$\varrho_{T}^{i}\ (i=1,2,\ldots,m)$ on~$\mathcal{B}_0$ such that
\begin{itemize}
\item[(i)]~$g_l$ is non-decreasing with respect to each
variable,~$g_l(0,\ldots,0)=0$ and~$g_l$ is continuous at~$(0,\ldots,0)$;
\item[(ii)]~$\varrho_{T}^{i}(i=1,2,\ldots,m)$ is precompact on $\mathcal{B}_0$, i.e., any sequence~$\{x_n\}\subseteq \mathcal{B}_0$ has a subsequence~$\{x_{n_k}\}$ which is Cauchy with
respect to~$\varrho_{T}^{i}$;
\item[(iii)]~the inequalities
\begin{equation}\label{20-8-3-20}
\begin{split}
&(d(S(T)y_1,S(T)y_2))^2\\\leq
&(d(y_1,y_2))^2+g_1\Big(\varrho_{T}^1\big(y_1,y_2\big),\varrho_{T}^2\big(y_1,y_2\big),\ldots,\varrho_{T}^m\big(y_1,y_2\big)\Big)
\end{split}
\end{equation}
and
\begin{equation}\label{20-7-27-25}
\begin{split}
&\left(d(S(T)y_1,S(T)y_2)\right)^2
\\ \leq &C\bigg[ (d(y_1,y_2))^2-(d(S(T)y_1,S(T)y_2))^2\\&+g_1\Big(\varrho_{T}^1\big(y_1,y_2\big),\varrho_{T}^2\big(y_1,y_2\big),\ldots,\varrho_{T}^m\big(y_1,y_2\big)\Big)\bigg]^{\beta}
\\&+g_2\Big(\varrho_{T}^1\big(y_1,y_2\big),\varrho_{T}^2\big(y_1,y_2\big),\ldots,\varrho_{T}^m\big(y_1,y_2\big)\Big)
\end{split}
\end{equation}
 hold for all~$y_1,y_2\in \mathcal{B}_0$ satisfying
~$\varrho_{T}^{i}(y_1,y_2)\leq \delta_0(i=1,2,\ldots,m)$.
\end{itemize}
Then there exists~$t_0>0$~such that for each bounded~$B\subseteq X$~the estimate
\begin{equation}\label{20-9-29-2}
\begin{split}
\alpha(S(t)B)\leq \Big\{(\alpha(\mathcal{B}_0))^{\frac{2(\beta-1)}{\beta}}+\frac{1-\beta}{T\beta(1+2C)^{\frac{1}{\beta}}}\big(t-t_{*}(B)-t_0-2T\big)\Big\}^{\frac{\beta}{2(\beta-1)}}
\end{split}
\end{equation}
holds for all~$t\geq t_0+2T+t_{*}(B)$,
where~$t_{*}(B)$ satisfies
\begin{equation*}
  S(t)B\subseteq \mathcal{B}_0, \ \forall t\geq t_{*}(B).
\end{equation*}
Thus,~$(X, \{S(t)\}_{t\geq0})$ possesses a polynomial attractor~$\mathcal{A}^*$ (see Definition~\ref{def21-1-31-1})
such that for every bounded set~$B\subseteq X$,
\begin{equation}
\begin{split}
\mathrm{ dist}\left(S(t)B, \mathcal{A}^*\right)\leq \Big\{(\alpha(\mathcal{B}_0))^{\frac{2(\beta-1)}{\beta}}+\frac{1-\beta}{T\beta(1+2C)^{\frac{1}{\beta}}}\big(t-t_0-2T-t_{*}(B)-1\big)\Big\}^{\frac{\beta}{2(\beta-1)}}
\end{split}
\end{equation}
holds for all~$t\geq t_0+2T+t_{*}(B)+1$.
\end{theorem}
\begin{proof}
For all~$y_1,y_2\in \mathcal{B}_0$~satisfying
~$\varrho_{T}^{i}(y_1,y_2)\leq \delta_0(i=1,2,\ldots,m)$, it follows from~$(\ref{20-7-27-25})$ that
\begin{equation*}
\begin{split}
\big(d(S(T)y_1,S(T)y_2)\big)^{\frac{2}{\beta}}
\leq &(2C)^{\frac{1}{\beta}}\bigg[ \big(d(y_1,y_2)\big)^2-\big(d(S(T)y_1,S(T)y_2)\big)^2\\&+g_1\Big(\varrho_{T}^1\big(y_1,y_2\big),\varrho_{T}^2\big(y_1,y_2\big),\ldots,\varrho_{T}^m\big(y_1,y_2\big)\Big)\bigg]
\\&+2^{\frac{1}{\beta}}g_2^{\frac{1}{\beta}}\Big(\varrho_{T}^1\big(y_1,y_2\big),\varrho_{T}^2\big(y_1,y_2\big),\ldots,\varrho_{T}^m\big(y_1,y_2\big)\Big),
\end{split}
\end{equation*}
which yields
\begin{equation}\label{21-4-18-5}
\begin{split}
&(2C)^{-\frac{1}{\beta}}\big(d(S(T)y_1,S(T)y_2)\big)^{\frac{2}{\beta}}+\big(d(S(T)y_1,S(T)y_2)\big)^2
\\ \leq & \big(d(y_1,y_2)\big)^2+g_1\Big(\varrho_{T}^1\big(y_1,y_2\big),\varrho_{T}^2\big(y_1,y_2\big),\ldots,\varrho_{T}^m\big(y_1,y_2\big)\Big)
\\&+C^{-\frac{1}{\beta}}g_2^{\frac{1}{\beta}}\Big(\varrho_{T}^1\big(y_1,y_2\big),\varrho_{T}^2\big(y_1,y_2\big),\ldots,\varrho_{T}^m\big(y_1,y_2\big)\Big).
\end{split}
\end{equation}
We rewrite $(\ref{21-4-18-5})$ as
\begin{equation}\label{21-4-18-6}
\begin{split}
&w\Big(\big(d(S(T)y_1,S(T)y_2)\big)^2\Big)
\\ \leq & \big(d(y_1,y_2)\big)^2+g_1\Big(\varrho_{T}^1\big(y_1,y_2\big),\varrho_{T}^2\big(y_1,y_2\big),\ldots,\varrho_{T}^m\big(y_1,y_2\big)\Big)
\\&+C^{-\frac{1}{\beta}}g_2^{\frac{1}{\beta}}\Big(\varrho_{T}^1\big(y_1,y_2\big),\varrho_{T}^2\big(y_1,y_2\big),\ldots,\varrho_{T}^m\big(y_1,y_2\big)\Big)
\end{split}
\end{equation}
with~$w(s)=(2C)^{-\frac{1}{\beta}}s^{\frac{1}{\beta}}+s,\ s\geq 0$. We denote by~$w^{-1}$
the inverse function of~$w$ on~$\mathds{R}^+$. Since~$w^{-1}$ is increasing, $(\ref{21-4-18-6})$ implies that
\begin{equation}\label{21-4-18-7}
\begin{split}
&\big(d(S(T)y_1,S(T)y_2)\big)^2
\\ \leq & w^{-1}\bigg( \big(d(y_1,y_2)\big)^2+g_1\Big(\varrho_{T}^1\big(y_1,y_2\big),\varrho_{T}^2\big(y_1,y_2\big),\ldots,\varrho_{T}^m\big(y_1,y_2\big)\Big)
\\&+C^{-\frac{1}{\beta}}g_2^{\frac{1}{\beta}}\Big(\varrho_{T}^1\big(y_1,y_2\big),\varrho_{T}^2\big(y_1,y_2\big),\ldots,\varrho_{T}^m\big(y_1,y_2\big)\Big)\bigg).
\end{split}
\end{equation}
Moreover, it is easy to check that~$w^{-1}(0)=0$ and $w^{-1}(s)<0,\ s>0$. Thus,  by Lemma~$\ref{lem21-4-18-1}$ we deduce from inequality~$(\ref{21-4-18-7})$ that~
\begin{equation}
\alpha(S(t)\mathcal{B}_0)\rightarrow 0\ \text{as}\ t\rightarrow+\infty.
\end{equation}
Consequently,  there exists~$t_0>0$~such that
\begin{equation}\label{20-7-27-5}
 \alpha(S(t)\mathcal{B}_0)<1,\ \forall t\geq t_0.
\end{equation}
For each fixed~$t\geq t_0$~and each~$\epsilon>0$,
by Definition~$\ref{def20-10-26-1}$, there exist sets~$F_1,F_2,\ldots,F_n$~such that
\begin{equation}\label{20-7-27-6}
  S(t)\mathcal{B}_0\subseteq\cup_{j=1}^{n}F_n,\ \mathrm{diam} F_j<\alpha(S(t)\mathcal{B}_0)+\epsilon.
\end{equation}
It follows from assumption (i) that there exists
~$\delta>0$~such that~$g_l(x_1,x_2,\ldots,x_m)<\epsilon$~$(l=1,2)$~whenever~$x_i\in [0,\delta]\ (i=1,2,\ldots,m)$. By the precompactness of
~$\varrho_{T}^i(i=1,2,\ldots,m)$, there exists a finite set ~$\mathcal{N}^i=\{x^i_{j}:j=1,2,\ldots,k_i\}\subseteq \mathcal{B}_0$~such that for every~$y\in \mathcal{B}_0$~there is~$x^i_{j}\in \mathcal{N}^i$~with the property ~$\varrho_{T}^i(S(t)y,S(t)x^i_{j})\leq\frac{1}{2}\min\{\delta,\delta_0\}$, i.e.,
\begin{equation}\label{20-7-27-9}
  S(t)\mathcal{B}_0\subseteq\cup_{j=1}^{k_i}C_{j}^{i},\ C_{j}^{i}=\big\{S(t)y:y\in \mathcal{B}_0,\ \varrho_{T}^i(S(t)y,S(t)x^i_{j})\leq\frac{1}{2}\min\{\delta,\delta_0\}\big\},\ i=1,\ldots,m.
\end{equation}
Consequently, we have
\begin{equation*}
  S(t)\mathcal{B}_0\subseteq\cup_{j_1,j_2,\ldots,j_m,j}( C_{j_1}^{1}\cap C_{j_2}^{2}\cap\ldots\cap C_{j_m}^{m}\cap F_j)
\end{equation*}
and
\begin{equation*}
  S(t+T)\mathcal{B}_0\subseteq\cup_{j_1,j_2,\ldots,j_m,j}\big(S(T)( C_{j_1}^{1}\cap C_{j_2}^{2}\cap\ldots\cap C_{j_m}^{m}\cap F_j)\big).
\end{equation*}
By~$(\ref{20-7-27-6})$~and~$(\ref{20-7-27-9})$, for any~$y_1,y_2\in\mathcal{B}_0$ such that $S(t)y_1,S(t)y_2\in C_{j_1}^{1}\cap C_{j_2}^{2}\cap\ldots\cap C_{j_m}^{m}\cap F_j$, we have
\begin{equation}\label{20-7-27-11}
  d\big(S(t)y_1,S(t)y_2\big)\leq \mathrm{diam} F_j<\alpha(S(t)\mathcal{B}_0)+\epsilon
\end{equation}
and
\begin{equation}\label{20-7-27-12}
  \varrho_{T}^i\big(S(t)y_1,S(t)y_2\big)\leq \min\{\delta,\delta_0\} (i=1,2,\ldots,m).
\end{equation}

Inequality~$(\ref{20-7-27-12})$~implies
\begin{equation}\label{20-7-27-13}
  g_l\Big(\varrho_{T}^1\big(S(t)y_1,S(t)y_2\big),\ldots,\varrho_{T}^m\big(S(t)y_1,S(t)y_2\big)\Big)<\epsilon,\ l=1,2.
\end{equation}
We deduce from~$(\ref{20-8-3-20})$,~$(\ref{20-7-27-25})$~and ~$(\ref{20-7-27-12})$~that
\begin{equation*}
\begin{split}
&\big(d(S(T+t)y_1,S(T+t)y_2)\big)^{\frac{2}{\beta}}\\\leq&(2C)^{\frac{1}{\beta}}\big[\big(d(S(t)y_1,S(t)y_2)\big)^2-\big(d(S(T+t)y_1,S(T+t)y_2)\big)^2\\&+g_1\big(\varrho_{T}^1(S(t)y_1,S(t)y_2),\ldots,\varrho_{T}^m(S(t)y_1,S(t)y_2)\big)\big]\\&+2^{\frac{1}{\beta}}g_2^{\frac{1}{\beta}}\big(\varrho_{T}^1(S(t)y_1,S(t)y_2),\ldots,\varrho_{T}^m(S(t)y_1,S(t)y_2)\big),
 \end{split}
\end{equation*}
which yields
\begin{equation}\label{20-8-3-31}
\begin{split}
&(2C)^{-\frac{1}{\beta}}\big(d(S(T+t)y_1,S(T+t)y_2)\big)^{\frac{2}{\beta}}+\big(d(S(T+t)y_1,S(T+t)y_2)\big)^{2}\\ \leq &\big(d(S(t)y_1,S(t)y_2)\big)^2+g_1\big(\varrho_{T}^1(S(t)y_1,S(t)y_2),\ldots,\varrho_{T}^m(S(t)y_1,S(t)y_2)\big)\\&+C^{-\frac{1}{\beta}}g_2^{\frac{1}{\beta}}\big(\varrho_{T}^1(S(t)y_1,S(t)y_2),\ldots,\varrho_{T}^m(S(t)y_1,S(t)y_2)\big),
 \end{split}
\end{equation}
i.e.,
\begin{equation}\label{20-8-3-32}
\begin{split}
&w\Big(\big(d(S(T+t)y_1,S(T+t)y_2\big)^{2}\Big)\\ \leq &\big(d(S(t)y_1,S(t)y_2)\big)^2+g_1\big(\varrho_{T}^1(S(t)y_1,S(t)y_2),\ldots,\varrho_{T}^m(S(t)y_1,S(t)y_2)\big)\\&+C^{-\frac{1}{\beta}}g_2^{\frac{1}{\beta}}\big(\varrho_{T}^1(S(t)y_1,S(t)y_2),\ldots,\varrho_{T}^m(S(t)y_1,S(t)y_2)\big).
 \end{split}
\end{equation}
Since~$w(s)$~is increasing on~$\mathds{R}^{+}$,~$(\ref{20-8-3-32})$~implies
\begin{equation}\label{20-9-14-9}
\begin{split}
&\big(d(S(T+t)y_1,S(T+t)y_2\big)^{2}\\ \leq &w^{-1}\bigg(\big(d(S(t)y_1,S(t)y_2)\big)^2+g_1\big(\varrho_{T}^1(S(t)y_1,S(t)y_2),\ldots,\varrho_{T}^m(S(t)y_1,S(t)y_2)\big)\\&+C^{-\frac{1}{\beta}}g_2^{\frac{1}{\beta}}\big(\varrho_{T}^1(S(t)y_1,S(t)y_2),\ldots,\varrho_{T}^m(S(t)y_1,S(t)y_2)\big).
 \bigg).
 \end{split}
\end{equation}
We derive from $(\ref{20-7-27-11})$,~$(\ref{20-7-27-13})$,~$(\ref{20-9-14-9})$ and the monotonically increasing property of~$w^{-1}$~that
\begin{equation*}
\begin{split}
\big(d(S(T+t)y_1,S(T+t)y_2\big)^{2} \leq w^{-1}\Big(\big(\alpha(S(t)\mathcal{B}_0)+\epsilon\big)^2+\epsilon+C^{-\frac{1}{\beta}}\epsilon^{\frac{1}{\beta}}\Big).
 \end{split}
\end{equation*}

As a consequence,
\begin{equation*}
\begin{split}
\big(\alpha(S(t+T)\mathcal{B}_0)\big)^2\leq w^{-1}\Big(\big(\alpha(S(t)\mathcal{B}_0)+\epsilon\big)^2+\epsilon+C^{-\frac{1}{\beta}}\epsilon^{\frac{1}{\beta}}\Big).
 \end{split}
\end{equation*}
Hence by the arbitrariness of~$\epsilon$, we have
\begin{equation}\label{20-8-3-44}
\begin{split}
w\Big(\big(\alpha(S(t+T)\mathcal{B}_0)\big)^2\Big)
\leq \big(\alpha(S(t)\mathcal{B}_0)\big)^2.
 \end{split}
\end{equation}
Inequality~$(\ref{20-8-3-44})$~is equivalent to
\begin{equation}\label{20-8-3-45}
\begin{split}
\big(\alpha(S(t+T)\mathcal{B}_0)\big)^{2}
\leq 2C\Big[\big(\alpha(S(t)\mathcal{B}_0)\big)^2-\big(\alpha(S(t+T)\mathcal{B}_0)\big)^2\Big]^{\beta}.
 \end{split}
\end{equation} Since~$\alpha(S(t+T)\mathcal{B}_0)\leq\alpha(S(t)\mathcal{B}_0)<1$~holds for all~$t\geq t_0$, we have
\begin{equation}\label{20-7-27-63}
(\alpha(S(t)\mathcal{B}_0))^2-(\alpha(S(t+T)\mathcal{B}_0))^2 \leq \big[(\alpha(S(t)\mathcal{B}_0))^2-(\alpha(S(t+T)\mathcal{B}_0))^2\big]^{\beta}.
\end{equation}
It follows from~$(\ref{20-8-3-45})$~and~$(\ref{20-7-27-63})$~that
\begin{equation*}
\begin{split}
(\alpha(S(t)\mathcal{B}_0))^2=&(\alpha(S(t+T)\mathcal{B}_0))^2+(\alpha(S(t)\mathcal{B}_0))^2-(\alpha(S(t+T)\mathcal{B}_0))^2\\
\leq&(1+ 2C)\big[(\alpha(S(t)\mathcal{B}_0))^2-(\alpha(S(t+T)\mathcal{B}_0))^2\big]^{\beta},
\end{split}
\end{equation*}
i.e.,
\begin{equation}\label{20-7-27-64}
\begin{split}
(\alpha(S(t)\mathcal{B}_0))^{\frac{2}{\beta}}
\leq(1+ 2C)^{\frac{1}{\beta}}\big[(\alpha(S(t)\mathcal{B}_0))^2-(\alpha(S(t+T)\mathcal{B}_0))^2\big]
\end{split}
\end{equation}
holds for all~$t\geq t_0$.
Since~$\mathcal{B}_0$~is positively invariant,~$\alpha(S(t)\mathcal{B}_0)$~is non-increasing with respect to~$t$. Therefore it follows from~$(\ref{20-7-27-64})$~that~$(\ref{20-7-26-41})$~holds with~$w(t)=(\alpha(S(t)\mathcal{B}_0))^2$,~$1+\alpha=\frac{1}{\beta}$~and~$h(t)=(1+ 2C)^{\frac{1}{\beta}}$. Consequently, by Lemma~$\ref{20-7-26-40}$, we have
\begin{equation*}
\begin{split}
\alpha(S(t)\mathcal{B}_0)\leq \Big\{(\alpha(\mathcal{B}_0))^{\frac{2(\beta-1)}{\beta}}+\frac{1-\beta}{T\beta(1+2C)^{\frac{1}{\beta}}}\big(t-t_0-2T\big)\Big\}^{\frac{\beta}{2(\beta-1)}},\ \forall t\geq t_0+2T,
\end{split}
\end{equation*}
which, together with lemma~$\ref{lemma 8-20-5-12}$, gives~$(\ref{20-9-29-2})$.

By Theorem~$\ref{20-8-5-3}$,~$(X, \{S(t)\}_{t\geq0})$~possesses a polynomial attractor~$\mathcal{A}^*$
such that for every bounded set~$B\subseteq X$,
\begin{equation}
\begin{split}
\mathrm{ dist}\left(S(t)B, \mathcal{A}^*\right)\leq \Big\{(\alpha(\mathcal{B}_0))^{\frac{2(\beta-1)}{\beta}}+\frac{1-\beta}{T\beta(1+2C)^{\frac{1}{\beta}}}\big(t-t_0-2T-t_{*}(B)-1\big)\Big\}^{\frac{\beta}{2(\beta-1)}}
\end{split}
\end{equation}
holds for all~$t\geq t_0+2T+t_{*}(B)+1$.
\end{proof}

\section{Existence of polynomial attractors and estimate of their attractive velocity for a class of wave equations}\label{sec20-8-10-89}

 We consider the following initial-boundary value problem:
\begin{align}
\label{wave equa}&u_{tt}-\Delta u+k||u_{t}||^p u_t+f(u)
=\displaystyle\int_{\Omega}K(x,y)u_{t}(y)dy+h(x) \ \ \text{in}\ [0,\infty)\times\Omega,\\
\label{boundary condition}&u=0 \ \text{on}\  [0,\infty)\times \partial\Omega,\\
\label{initial condition}&u(x,0)=u_0(x) , u_{t}(x,0)=u_{1}(x),\  x\in\Omega,
\end{align}
where~$k$~and~$p$~are positive constants,~$K\in L^2(\Omega\times\Omega)$,~$h\in L^2(\Omega)$.~$f\in C^1(\mathds{R})$~satisfies the following polynomial growth condition: there exists a positive constant~$M$~such that
\begin{equation}\label{growth}
    |f'(s)| \leq M|s|^q \ \text{for all}\   |s| \geq 1,
  \end{equation}
where~$0<q <\displaystyle\frac{2}{n-2}$~if~$n\geq3$ and~$0<q<\infty$~if~$n\leq2$. Moreover, the following dissipativity condition holds:
\begin{equation}\label{dissipativity condition}
    \liminf_{|s|\rightarrow\infty}f'(s)\equiv \mu > -\lambda_{1},
  \end{equation}
where~$\lambda_{1}$~is the first eigenvalue of the operator~$-\Delta$~equipped with Dirichlet boundary condition.

Recently, we have proved the following result in \cite{my1}.
\begin{lemma}\cite{my1}\label{20-9-30-1}
Let~$T> 0$~be arbitrary. Under the above assumptions, for every~$(u_0,u_1) \in H^1_0(\Omega)\times L^2(\Omega)$, the initial
boundary value problem~$(\ref{wave equa})$-$(\ref{initial condition})$~has a unique weak solution~$u \in C([0,T];H^1_0(\Omega)) \cap C^1([0,T];L^2(\Omega))$, which generates the semigroup
\begin{equation*}
  S(t)(u_0,u_1) = (u(t),u_t (t)),\ t\geq0
\end{equation*}
on~$H^1_0(\Omega)\times L^2(\Omega) $.

Furthermore, the semigroup~$\{S(t)\}_{t\geq0}$~is dissipative, which implies the existence of a positively invariant bounded absorbing set~$\mathcal{B}_0$.
\end{lemma}

To establish the main result in this section, the following lemmas are also needed.
\begin{lemma}\cite{MR4064014}\label{lemma l2.3}
Let~$\left(H,(\cdot,\cdot)_H\right)$~be an inner product space with the induced norm~$\|\cdot\|_H$~and constant~$p>1$. Then there exists some positive constant~$C_p$~such that for any~$x,y\in H$~satisfying~$(x,y)\neq (0,0)$, we have
\begin{eqnarray}\label{2.11}
&\big(\|x\|_H^{p-2}x-\|y\|_H^{p-2}y,x-y\big)_H
&\geq
 \begin{cases}
 C_p\|x-y\|_H^{p}, ~  p\geq2;\\
C_p\displaystyle\frac{\|x-y\|_H^{2}}{(\|x\|_H+\|y\|_H)^{2-p}},~1<p<2.
 \end{cases}
\end{eqnarray}
\end{lemma}

\begin{lemma}\citep[Corollary 4]{Simon1986}\label{lemma 1-11-1}
Assume~$X\hookrightarrow\hookrightarrow B \hookrightarrow Y$~where~$X, B, Y$~are Banach spaces. The following statements hold.
\begin{enumerate}[(i)]
\item  Let~$F$~be bounded in~$L^{p}(0,T;X)$~where~$1\leq p<\infty$, and~$\partial F/\partial t=\{\partial f/\partial t:f\in F\}$~be bounded in $L^{1}(0,T;Y)$, where~$\partial /\partial t$~is the weak time derivative. Then~$F$~is relatively compact in~$L^{p}(0,T;B)$.
  \item Let~$F$~be bounded in~$L^{\infty}(0,T;X)$~and~$\partial F/\partial t$~be bounded in~$L^{r}(0,T;Y)$~where~$r>1$. Then~$F$~is relatively compact in~$C(0,T;B)$.
\end{enumerate}
\end{lemma}

\begin{lemma}\cite{Lax}\label{lemma 10-10-3}
The integral operator
\begin{eqnarray*}
K:L^{2}(\Omega  )&\longrightarrow &L^{2}(\Omega) \\
v&\longmapsto &\int_{\Omega}K(x,y)v(y)dy
\end{eqnarray*}
is a compact operator provided that the kernel~$K(x,y)$~is square-integrable.
\end{lemma}
Now we will apply Theorem~$\ref{20-7-27-80}$~to problem~$(\ref{wave equa})$-$(\ref{initial condition})$.

\begin{theorem}\label{polynomialdissipativity}
Under conditions~$(\ref{growth})$ and~$(\ref{dissipativity condition})$, the dynamical system $(H_0^1(\Omega)\times L^2(\Omega),\{S(t)\}_{t\geq0})$\ generated by problem~$(\ref{wave equa})$-$(\ref{initial condition})$~is polynomially decaying with respect to noncompactness measure. More precisely, there exists~$t_0>0$~such that for any bounded~$B\subseteq H_0^1(\Omega)\times L^2(\Omega)$, we have
\begin{equation*}
\begin{split}
\alpha(S(t)B)\leq &\Big\{(\alpha(\mathcal{B}_0))^{-p}+\frac{pkC_p}{2^{p+2}}\big(t-t_0-t_{*}(B)\big)\Big\}^{-\frac{1}{p}},\ \forall t\geq t_0+t_{*}(B),
\end{split}
\end{equation*}
where~$\mathcal{B}_0\subseteq H_{0}^{1}(\Omega)\times L^{2}(\Omega)$ is a positively invariant bounded absorbing set and~$t_{*}(B)$~is the entering time of~$B$~into~$\mathcal{B}_0$.

 Thus, this dynamical system possesses a polynomial attractor such that for every bounded set~$B\subseteq H_{0}^{1}(\Omega)\times L^{2}(\Omega)$,
\begin{equation}\label{21-4-24-1}
\begin{split}
\mathrm{ dist}\left(S(t)B, \mathcal{A}^*\right)\leq  &\Big\{(\alpha(\mathcal{B}_0))^{-p}+\frac{pkC_p}{2^{p+2}}\big(t-t_0-t_{*}(B)-1\big)\Big\}^{-\frac{1}{p}}
\end{split}
\end{equation}
holds for all~$t\geq t_0+t_{*}(B)+1$.
\end{theorem}
\begin{proof}
By Lemma~$\ref{20-9-30-1}$,~$(H_0^1(\Omega)\times L^2(\Omega),\{S(t)\}_{t\geq0})$~is dissipative. Let~$\mathcal{B}_0\subseteq H_{0}^{1}(\Omega)\times L^{2}(\Omega)$~be a positively invariant bounded absorbing set.

Write~$\Psi(u_t(t,x))=\int_{\Omega}K(x,y)u_t(t,y)dy$.
Let~$w(t),v(t)$~be two weak solutions to~$(\ref{wave equa})$-$(\ref{initial condition})$~corresponding to initial
data~$y_1,y_2\in \mathcal{B}_0$:
\begin{equation*}
(w(t),w_t(t))\equiv S(t)y_1, (v(t),v_t(t))\equiv S(t)y_2, \ y_1,y_2\in \mathcal{B}_0.
\end{equation*}
Since~$\mathcal{B}_0$~is positively invariant, we have
\begin{equation}\label{20-7-25-17}
\begin{cases}\|(w(t),w_t(t))\|_{ H^1_0(\Omega)\times L^2(\Omega)}\leq C,\\\|(v(t),v_t(t))\|_{ H^1_0(\Omega)\times L^2(\Omega)}\leq C,\end{cases}\forall t>0, y_1,y_2\in \mathcal{B}_0.
\end{equation}
The difference~$z(t)=w(t)-v(t)$~satisfies
\begin{equation}\label{20-7-25-1}
z_{tt}-\Delta z +k(||w_t||^p w_t-||v_t||^p v_t)+f(w)-f(v)=\Psi(z_t).
 \end{equation}
Multiplying~$(\ref{20-7-25-1})$~by~$z_t$~in~$L^2(\Omega)$, we obtain
\begin{equation}\label{20-7-25-2}
\begin{split}
&\frac{1}{2}\frac{d}{dt}(||z_t||^2+||\nabla z||^2) +k(||w_t||^p w_t-||v_t||^p v_t,z_t)\\=&(\Psi(z_t),z_t)-(f(w)-f(v),z_t).
\end{split}
 \end{equation}
Let~$T$~be an arbitrary positive constant.
Integrating~$(\ref{20-7-25-2})$~from~$t$~to~$T$, we obtain
 \begin{equation}\label{20-7-25-3}
\begin{split}
E_z(t)=&E_z(T)+k\int_t^T(||w_t||^p w_t-||v_t||^p v_t,z_t)d\tau\\&+\int_t^T(f(w)-f(v),z_t)d\tau-\int_t^T(\Psi(z_t),z_t)d\tau,
\end{split}
 \end{equation}
 where~$E_z(t)=\frac{1}{2}(||z_t(t)||^2+||\nabla z(t)||^2)$.

Integrating~$(\ref{20-7-25-3})$~from~$0$~to~$T$~yields
 \begin{equation}\label{20-7-25-4}
\begin{split}
TE_z(T)=&\int_0^TE_z(t)dt-\int_0^T\int_t^Tk(||w_t||^p w_t-||v_t||^p v_t,z_t)d\tau dt\\&-\int_0^T\int_t^T(f(w)-f(v),z_t)d\tau dt +\int_0^T\int_t^T(\Psi(z_t),z_t)d\tau dt.
\end{split}
 \end{equation}
Multiplying~$(\ref{20-7-25-1})$~by~$z$~in~$L^2(\Omega)$~gives
 \begin{equation}\label{20-7-25-5}
\begin{split}
\frac{1}{2}(||z_t||^2+||\nabla z||^2) =&-\frac{1}{2}\frac{d}{dt}(z_t,z)+||z_t||^2-\frac{1}{2}(f(w)-f(v),z)\\&+\frac{1}{2}(\Psi(z_t),z)-\frac{k}{2}(||w_t||^p w_t-||v_t||^p v_t,z).
\end{split}
\end{equation}
Integrating~$(\ref{20-7-25-5})$~from~$0$~to~$T$~yields
 \begin{equation}\label{20-7-25-6}
\begin{split}
\int_0^T E_z(t)dt=&-\frac{1}{2}(z_t,z)|^T_0+ \int_0^T ||z_t||^2 dt- \frac{1}{2} \int_0^T(f(w)-f(v),z) dt \\&+  \frac{1}{2} \int_0^T(\Psi(z_t),z)  dt- \frac{k}{2} \int_0^T(||w_t||^p w_t-||v_t||^p v_t,z)dt.
\end{split}
\end{equation}
Substituting~$(\ref{20-7-25-6})$~into~$(\ref{20-7-25-4})$, we have
 \begin{equation}\label{20-7-25-7}
\begin{split}
TE_z(T)=&-\frac{1}{2}(z_t,z)|^T_0+ \int_0^T ||z_t||^2 dt- \frac{1}{2} \int_0^T(f(w)-f(v),z) dt \\&+  \frac{1}{2} \int_0^T(\Psi(z_t),z)  dt- \frac{k}{2} \int_0^T(||w_t||^p w_t-||v_t||^p v_t,z)dt\\&-\int_0^T\int_t^Tk(||w_t||^p w_t-||v_t||^p v_t,z_t)d\tau dt\\&-\int_0^T\int_t^T(f(w)-f(v),z_t)d\tau dt +\int_0^T\int_t^T(\Psi(z_t),z_t)d\tau dt.
\end{split}
 \end{equation}
Using~$(\ref{20-7-25-17})$, we estimate the terms on the right hand side of identity~$(\ref{20-7-25-7})$~as follows.

By~$(\ref{growth})$, we have
\begin{equation}\label{20-11-12-5}
\begin{split}
 &\left[\int_{\Omega}\big(f(w)-f(v)\big)^{2}dx\right]^{1/2} \\ =&\left\{\int_{\Omega}\left[\int_{0}^{1}f'\big(v+\theta(w-v)\big)(w-v)d\theta\right]^{2}dx\right\}^{1/2}\\
   \leq & \left\{\int_{\Omega}\left[\int_{0}^{1}M\Big(|v+\theta(w-v)|^{q}+1\Big)|w-v|d\theta\right]^{2}dx\right\}^{1/2} \\
   \leq & C\left\{\int_{\Omega}\big[(|v|^{q}+|w|^{q}+1)|w-v|\big]^{2}dx\right\}^{1/2}\\
   \leq & C\left\{\int_{\Omega}(|v|^{2q}+|w|^{2q}+1)|w-v|^{2}dx\right\}^{1/2} \\
   \leq &  C\left\{\left(\int_{\Omega}|v|^{2q}|w-v|^{2}dx\right)^{1/2}+\left(\int_{\Omega}|w|^{2q}|w-v|^{2}dx\right)^{1/2}+\left(\int_{\Omega}|w-v|^{2}dx\right)^{1/2}\right\}.
\end{split}\end{equation}
When~$n\geq3$, we take~$r=\frac{n}{(n-2)q}$,~$\overline{r}=\frac{n}{n-(n-2)q}$, then by~$q<\frac{2}{n-2}$, it is obvious that ~$\displaystyle\frac{1}{r}+\displaystyle\frac{1}{\overline{r}}=1$, $2qr=\displaystyle\frac{2n}{n-2}$~and ~$2\overline{r}<\displaystyle\frac{2n}{n-2}$. When~$n\leq2$, we take~$r=\overline{r}=2$. Thus for any~$n\in \mathds{N}$~we have~$H_{0}^{1}(\Omega)\hookrightarrow L^{2qr}(\Omega)$~and~$H_{0}^{1}(\Omega)\hookrightarrow \hookrightarrow L^{2\overline{r}}(\Omega)$.

Hence, combining~$(\ref{20-7-25-17})$~and~$(\ref{20-11-12-5})$~gives
\begin{equation}\label{20-7-25-30}
\begin{split}
 &\left[\int_{\Omega}\big(f(w)-f(v)\big)^{2}dx\right]^{1/2} \\
 \leq&C\Bigg\{\left(\int_{\Omega}|w|^{2qr}dx\right)^{\frac{1}{2r}}\left(\int_{\Omega}|w-v|^{2\overline{r}}dx\right)^{\frac{1}{2\overline{r}}}\\&+\left(\int_{\Omega}|v|^{2qr}dx\right)^{\frac{1}{2r}}\left(\int_{\Omega}|w-v|^{2\overline{r}}dx\right)^{\frac{1}{2\overline{r}}} +\left(\int_{\Omega}|w-v|^{2}dx\right)^{1/2}\Bigg\}\\
   \leq &C\cdot\Big( \|\nabla v\|^{q}+\|\nabla w\|^{q}+1\Big) \|w-v\|_{2\overline{r}}
   \\
   \leq &C\|w-v\|_{2\overline{r}}.
\end{split}\end{equation}
Consequently, we have
    \begin{equation}\label{20-7-25-12}
\begin{split}
-\int_0^T\int_t^T(f(w)-f(v),z_t)d\tau dt\leq & T^2\sup_{t\in[0,T]}\|f(w(t))-f(v(t))\|\sup_{t\in[0,T]}\|z_t(t)\|\\
\leq & T^2C\sup_{t\in[0,T]}\|f(w(t))-f(v(t))\|\\
\leq &T^2C\sup_{t\in[0,T]}\|w(t)-v(t)\|_{2\overline{r}}
\end{split}
 \end{equation}
and
  \begin{equation}\label{20-7-25-8}
\begin{split}
- \frac{1}{2} \int_0^T(f(w)-f(v),z) dt &\leq \frac{T}{2}\sup_{t\in[0,T]}\|f(w(t))-f(v(t))\|\sup_{t\in[0,T]}\|z(t)\|\\
&\leq TC\sup_{t\in[0,T]}\|z(t)\|.
\end{split}
 \end{equation}
By Lemma~$\ref{lemma l2.3}$~we have
 \begin{equation*}
(||w_t||^p w_t-||v_t||^p v_t,w_t-z_t)\geq C_p||w_t-v_t||^{p+2},
 \end{equation*}
 thus
  \begin{equation}\label{20-8-2-21}
||w_t-v_t||^{2}\leq C_p^{-\frac{2}{p+2}} (||w_t||^p w_t-||v_t||^p v_t,w_t-v_t)^{\frac{2}{p+2}}.
 \end{equation}
We deduce from~$(\ref{20-8-2-21})$~and the concavity of~$g(s)=s^{\frac{2}{p+2}}(s>0)$~that
   \begin{equation}\label{20-7-25-14}
   \begin{split}
    \int_0^T ||z_t||^2 dt\leq &C_p^{-\frac{2}{p+2}}\int_0^T (||w_t||^p w_t-||v_t||^p v_t,w_t-v_t)^{\frac{2}{p+2}}dt\\
    \leq &C_p^{-\frac{2}{p+2}}T \Big(\frac{1}{T}\int_0^T(||w_t||^p w_t-||v_t||^p v_t,w_t-v_t)dt\Big)^{\frac{2}{p+2}}\\
    =&C_p^{-\frac{2}{p+2}}T^{\frac{p}{p+2}}\Big(\int_0^T(||w_t||^p w_t-||v_t||^p v_t,w_t-v_t)dt\Big)^{\frac{2}{p+2}}.
 \end{split}\end{equation}
Taking~$t=0$~in~$(\ref{20-7-25-3})$~yields
 \begin{equation}\label{20-7-25-15}
\begin{split}
&\int_0^T(||w_t||^p w_t-||v_t||^p v_t,w_t-v_t)dt\\=&\frac{1}{k}\left(E_z(0)-E_z(T)-\int_0^T(f(w)-f(v),z_t)d\tau+\int_0^T(\Psi(z_t),z_t)d\tau\right).
\end{split}\end{equation}
We infer from~$(\ref{20-7-25-17})$,~$(\ref{20-7-25-30})$,~$(\ref{20-7-25-14})$~and~$(\ref{20-7-25-15})$~that
 \begin{equation}\label{20-7-25-16}
   \begin{split}
    \int_0^T ||z_t||^2 dt\leq &(kC_p)^{-\frac{2}{p+2}}T^{\frac{p}{p+2}}\bigg(E_z(0)-E_z(T)\\&-\int_0^T(f(w)-f(v),z_t)d\tau+\int_0^T(\Psi(z_t),z_t)d\tau\bigg)^{\frac{2}{p+2}}\\
    \leq&(kC_p)^{-\frac{2}{p+2}}T^{\frac{p}{p+2}}\bigg(E_z(0)-E_z(T)\\&+TC\sup_{t\in[0,T]}\|w(t)-v(t)\|_{2\overline{r}}+ C\int_0^T\|\Psi(z_t(t))\|dt\bigg)^{\frac{2}{p+2}}.
 \end{split}\end{equation}
We proceed to give the estimates for the rest terms:
 \begin{equation}\label{20-7-27-70}
\begin{split}
-\frac{1}{2}(z_t,z)|^T_0\leq &\frac{1}{2}\big(\|z_t(T)\|\|z(T)\|+\|z_t(0)\|\|z(0)\|\big)\\
\leq &C\big(\|z(T)\|+\|z(0)\|\big)\\
\leq &C\sup_{t\in[0,T]}\|z(t)\|,
\end{split}
 \end{equation}
  \begin{equation}\label{20-7-25-9}
\begin{split}
 \frac{1}{2} \int_0^T(\Psi(z_t),z)  dt&\leq \frac{T}{2}\sup_{t\in[0,T]}\|\Psi(z_t(t))\|\sup_{t\in[0,T]}\|z(t)\|\\
 &\leq \frac{T}{2}\|K\|_{L^2(\Omega\times\Omega)}\sup_{t\in[0,T]}\|z_t(t)\|\sup_{t\in[0,T]}\|z(t)\|\\
&\leq TC\sup_{t\in[0,T]}\|z(t)\|,
\end{split}
 \end{equation}
  \begin{equation}\label{20-7-25-10}
\begin{split}
- \frac{k}{2} \int_0^T(||w_t||^p w_t-||v_t||^p v_t,z)dt\leq TC\sup_{t\in[0,T]}\|z(t)\|,
\end{split}
 \end{equation}
   \begin{equation}\label{20-7-25-11}
\begin{split}
-\int_0^T\int_t^Tk(||w_t||^p w_t-||v_t||^p v_t,z_t)d\tau dt\leq0,
\end{split}
 \end{equation}
  \begin{equation}\label{20-7-25-13}
\begin{split}
\int_0^T\int_t^T(\Psi(z_t),z_t)d\tau dt\leq&\int_0^T\int_0^T\left|(\Psi(z_t),z_t)\right|d\tau dt\\ \leq & T\sup_{t\in[0,T]}\|z_t(t)\|\int_0^T\|\Psi(z_t(t))\|dt\\
\leq &TC\int_0^T\|\Psi(z_t(t))\|dt.
\end{split}
 \end{equation}
Plugging~$(\ref{20-7-25-12})$,~$(\ref{20-7-25-8})$,~$(\ref{20-7-25-16})$,~$(\ref{20-7-27-70})$,~$(\ref{20-7-25-9})$,~$(\ref{20-7-25-10})$,~$(\ref{20-7-25-11})$~and~$(\ref{20-7-25-13})$~ into~$(\ref{20-7-25-7})$, we obtain
  \begin{equation}\label{20-7-26-1}
   \begin{split}
   E_z(T)\leq&C_{T}\Big(\sup_{t\in[0,T]}\|z(t)\|+\sup_{t\in[0,T]}\|z(t)\|_{2\overline{r}}+\int_0^T\|\Psi(z_t(t))\|dt\Big)\\&+ (kTC_p)^{-\frac{2}{p+2}}\bigg(E_z(0)-E_z(T)\\&+TC\sup_{t\in[0,T]}\|z(t)\|_{2\overline{r}}+ C\int_0^T\|\Psi(z_t(t))\|dt\bigg)^{\frac{2}{p+2}}\\\leq&C_{T}\Big(\sup_{t\in[0,T]}\|z(t)\|_{2\overline{r}}+\int_0^T\|\Psi(z_t(t))\|dt\Big)\\&+ (kTC_p)^{-\frac{2}{p+2}}\bigg(E_z(0)-E_z(T)\\&+TC\sup_{t\in[0,T]}\|z(t)\|_{2\overline{r}}+ C\int_0^T\|\Psi(z_t(t))\|dt\bigg)^{\frac{2}{p+2}}.
 \end{split}\end{equation}
 Since~$H_{0}^{1}(\Omega)\hookrightarrow \hookrightarrow L^{2\overline{r}}(\Omega)\hookrightarrow L^{2}(\Omega)$, it follows from Lemma~$\ref{lemma 1-11-1}$ and inequality~$(\ref{20-7-25-17})$ that~$\rho_{T}(y_1,y_2)=\sup_{t\in[0,T]}\|z(t)\|_{2\overline{r}}$ is precompact on~$\mathcal{B}_0$.

Let~$\mathcal{A}$~denote the strictly positive operator on~$L^2(\Omega)$~defined by~$\mathcal{A}=-\triangle$~with domain~$D(\mathcal{A})=H^2(\Omega)\cap H^1_0(\Omega)$.
Let~$V$~be the completion of~$L^2(\Omega)$~with respect to the norm~$\|\cdot\|_{V}$~given by~$\|\cdot\|_V=\|\Psi(\cdot)\|+\|\mathcal{A}^{-\frac{1}{2}}\cdot\|$~and~$W$~be the completion of~$L^2(\Omega)$~with respect to the norm~$\|\cdot\|_{W}$~given by~$\|\cdot\|_W=\|\mathcal{A}^{-\frac{1}{2}}\cdot\|$.

By Lemma~$\ref{lemma 10-10-3}$, we have
\begin{equation}\label{20-7-27-66}
 L^2(\Omega)\hookrightarrow\hookrightarrow V\hookrightarrow W.
\end{equation}
It follows from~$(\ref{20-7-25-30})$~that
\begin{equation}\label{20-7-27-67}
\begin{split}
  \|f(w(t))\|\leq C.
\end{split}\end{equation}
In addition, we have
\begin{equation}\label{20-7-27-68}
\begin{split}
\|\Psi(w_{t}(t))\|\leq \|K\|_{L^2(\Omega\times\Omega)}\|w_{t}(t)\|\leq C.
\end{split}\end{equation}
We infer from~$(\ref{wave equa})$,~$(\ref{20-7-27-67})$~and~$(\ref{20-7-27-68})$~that
\begin{equation*}
\begin{split}
  \|\mathcal{A}^{-\frac{1}{2}}w_{tt}(t)\|\leq&\|\nabla w(t)\|+k\|w_{t}(t)\|^p\|\mathcal{A}^{-\frac{1}{2}}w_{t}(t)\|\\
  &+\|\mathcal{A}^{-\frac{1}{2}}\big(\Psi(w_{t}(t))+h-f(w(t))\big)\|\\ \leq &C.
  \end{split}
\end{equation*}
Hence,
\begin{equation}\label{20-7-27-71}
   \int_0^T\|\mathcal{A}^{-\frac{1}{2}}w_{tt}(t)\|dt\leq C_{T}.
\end{equation}
Besides, we have
\begin{equation}\label{20-7-27-77}
  \int_0^T\|w_{t}(t)\|dt\leq C_{T}.
\end{equation}
By Lemma~$\ref{lemma 1-11-1}$,~$(\ref{20-7-27-66})$,~$(\ref{20-7-27-71})$~and~$(\ref{20-7-27-77})$~imply that ~$\varrho_{T}(y_1,y_2)=\int_0^T\|\Psi(z_t(t))\|dt$~is precompact on~$\mathcal{B}_0$.

Thus by Theorem~$\ref{20-7-27-80}$, we deduce from~$(\ref{20-7-26-1})$~that there exists~$t_0>0$~such that for any bounded~$B\subseteq X$,
\begin{equation}\label{20-7-28-1}
\begin{split}
\alpha(S(t)B)\leq &\Big\{(\alpha(\mathcal{B}_0))^{-p}+\frac{p}{2\big(T^{\frac{2}{p+2}}+2^{\frac{p}{p+2}+1}(kC_p)^{-\frac{2}{p+2}}\big)^{\frac{p+2}{2}}}\big(t-t_0-2T-t_{*}(B)\big)\Big\}^{-\frac{1}{p}}
\end{split}
\end{equation}
holds for all~$t\geq t_0+2T+t_{*}(B)$,
where~$t_{*}(B)$~satisfies
\begin{equation*}
  S(t)B\subseteq \mathcal{B}_0, \ \forall t\geq t_{*}(B).
\end{equation*}
Since~$\Big\{(\alpha(\mathcal{B}_0))^{-p}+\frac{p}{2\big(T^{\frac{2}{p+2}}+2^{\frac{p}{p+2}+1}(kC_p)^{-\frac{2}{p+2}}\big)^{\frac{p+2}{2}}}\big(t-t_0-2T-t_{*}(B)\big)\Big\}^{-\frac{1}{p}}$~is continuous and increasing with respect to~$T$, where~$T$~is an arbitrary positive constant, by taking~$T\rightarrow0$~in~$(\ref{20-7-28-1})$~we have
\begin{equation*}
\begin{split}
\alpha(S(t)B)\leq &\Big\{(\alpha(\mathcal{B}_0))^{-p}+\frac{pkC_p}{2^{p+2}}\big(t-t_0-t_{*}(B)\big)\Big\}^{-\frac{1}{p}}
\end{split}
\end{equation*}
for all~$t> t_0+t_{*}(B)$.

 Thus, the dynamical system~$(H_0^1(\Omega)\times L^2(\Omega),\{S(t)\}_{t\geq0})$ generated by problem~$(\ref{wave equa})$-$(\ref{initial condition})$~possesses a polynomial attractor such that for every bounded set~$B\subseteq H_{0}^{1}(\Omega)\times L^{2}(\Omega)$,
\begin{equation}
\begin{split}
\mathrm{ dist}\left(S(t)B, \mathcal{A}^*\right)\leq  &\Big\{(\alpha(\mathcal{B}_0))^{-p}+\frac{pkC_p}{2^{p+2}}\big(t-t_0-t_{*}(B)-1\big)\Big\}^{-\frac{1}{p}}
\end{split}
\end{equation}
holds for all~$t\geq t_0+t_{*}(B)+1$.
The proof is completed.
 \end{proof}
 \begin{remark}
As can be seen from formula $(\ref{21-4-24-1})$, when the kinetic energy is very small, the smaller~$p$  is, the faster the energy dissipates, and the higher the attractive speed of~$\mathcal{A}^*$ is. When~$p\rightarrow 0$, the damping  approaches linear damping, and the attractive speed of~$\mathcal{A}^*$ approaches exponential speed. This result is consistent with our intuitive prediction.
\end{remark}




 \section*{References }
\bibliographystyle{plain}



%


\end{document}